\documentclass[11pt]{amsart}
\usepackage{amssymb}

\newcommand{\dbar}{\ensuremath{\overline\partial}}
\newcommand{\dbarstar}{\ensuremath{\overline\partial^*}}
\newcommand{\C}{\ensuremath{\mathbb{C}}}
\newcommand{\R}{\ensuremath{\mathbb{R}}}

\makeatletter
\newcommand{\sumprime}{\if@display\sideset{}{'}\sum%
            \else\sum'\fi}
\makeatother

\begin{document}

\numberwithin{equation}{section}

\newtheorem{theorem}{Theorem}[section]
\newtheorem{proposition}[theorem]{Proposition}
\newtheorem{conjecture}[theorem]{Conjecture}
\def\theconjecture{\unskip}
\newtheorem{corollary}[theorem]{Corollary}
\newtheorem{lemma}[theorem]{Lemma}
\newtheorem{observation}[theorem]{Observation}
\theoremstyle{definition}
\newtheorem{definition}{Definition}
\numberwithin{definition}{section}
\newtheorem{remark}{Remark}
\def\theremark{\unskip}
\newtheorem{question}{Question}
\def\thequestion{\unskip}
\newtheorem{example}{Example}
\def\theexample{\unskip}
\newtheorem{problem}{Problem}

\def\vvv{\ensuremath{\mid\!\mid\!\mid}}
\def\intprod{\mathbin{\lr54}}
\def\reals{{\mathbb R}}
\def\integers{{\mathbb Z}}
\def\N{{\mathbb N}}
\def\complex{{\mathbb C}\/}
\def\dist{\operatorname{dist}\,}
\def\spec{\operatorname{spec}\,}
\def\interior{\operatorname{int}\,}
\def\trace{\operatorname{tr}\,}
\def\cl{\operatorname{cl}\,}
\def\essspec{\operatorname{esspec}\,}
\def\range{\operatorname{\mathcal R}\,}
\def\kernel{\operatorname{\mathcal N}\,}
\def\dom{\operatorname{\mathcal D}\,}
\def\linearspan{\operatorname{span}\,}
\def\lip{\operatorname{Lip}\,}
\def\sgn{\operatorname{sgn}\,}
\def\Z{ {\mathbb Z} }
\def\e{\varepsilon}
\def\p{\partial}
\def\rp{{ ^{-1} }}
\def\Re{\operatorname{Re\,} }
\def\Im{\operatorname{Im\,} }
\def\dbarb{\bar\partial_b}
\def\eps{\varepsilon}
\def\O{\Omega}
\def\Lip{\operatorname{Lip\,}}

\def\Hs{{\mathcal H}}
\def\E{{\mathcal E}}
\def\scriptu{{\mathcal U}}
\def\scriptr{{\mathcal R}}
\def\scripta{{\mathcal A}}
\def\scriptc{{\mathcal C}}
\def\scriptd{{\mathcal D}}
\def\scripti{{\mathcal I}}
\def\scriptk{{\mathcal K}}
\def\scripth{{\mathcal H}}
\def\scriptm{{\mathcal M}}
\def\scriptn{{\mathcal N}}
\def\scripte{{\mathcal E}}
\def\scriptt{{\mathcal T}}
\def\scriptr{{\mathcal R}}
\def\scripts{{\mathcal S}}
\def\scriptb{{\mathcal B}}
\def\scriptf{{\mathcal F}}
\def\scriptg{{\mathcal G}}
\def\scriptl{{\mathcal L}}
\def\scripto{{\mathfrak o}}
\def\scriptv{{\mathcal V}}
\def\frakg{{\mathfrak g}}
\def\frakG{{\mathfrak G}}

\def\ov{\overline}

\author{Siqi Fu}
\thanks
{The author was supported in part by NSF grant DMS-0805852.}
\address{Department of Mathematical Sciences,
Rutgers University-Camden, Camden, NJ 08102}
\email{sfu@camden.rutgers.edu}
\title[]  
{Positivity of the $\bar\partial$-Neumann Laplacian}
\maketitle

\leftline{\it Dedicated to Professor Linda Rothschild}

\bigskip

\begin{abstract}
We study the $\bar\partial$-Neumann Laplacian from spectral theoretic perspectives.
In particular, we show how pseudoconvexity of a bounded domain is characterized by positivity of the $\bar\partial$-Neumann Laplacian.

\bigskip
\noindent{{\sc Mathematics Subject Classification} (2000): 32W05.}

\smallskip

\noindent{{\sc Keywords}: Pseudoconvex, $\bar\partial$-Neumann Laplacian, Dolbeault cohomology, $L^2$-cohomology.}
\end{abstract}

\bigskip



\section{Introduction}\label{sec:intro}

Whether or not a given system has positive ground state energy is
a widely studied problem with significant repercussions in physics, particularly in quantum mechanics. It follows from the classical Hardy inequality that the bottom of the spectrum
of the Dirichlet Laplacian on a domain in $\R^n$ that satisfies the outer cone
condition is positive if and only if its inradius is finite (see~\cite{Davies95}).
Whereas spectral behavior of the Dirichlet Laplacian is insensitive to boundary
geometry, the story for the $\bar\partial$-Neumann Laplacian is different.
Since the work of Kohn \cite{Kohn63, Kohn64} and H\"{o}rmander \cite{Hormander65},
it has been known that existence and regularity of the $\bar\partial$-Neumann Laplacian
closely depend on the underlying geometry (see the surveys \cite{BoasStraube99, Christ99, DangeloKohn99, FuStraube01} and the monographs \cite{ChenShaw99, Straube09}).

Let $\Omega$ be a domain in $\C^n$.  It follows from the classical Theorem B of Cartan
that if $\Omega$ is pseudoconvex, then the Dolbeault cohomology groups $H^{0, q}(\Omega)$
vanish for all $q\ge 1$. (More generally, for any coherent analytic sheaf $\scriptf$ over
a Stein manifold, the sheaf cohomology groups $H^q(X, \scriptf)$ vanish for all $q\ge 1$.) The converse is also true (\cite{Serre53}, p.~65). Cartan's Theorem B and its converse were generalized by Laufer \cite{Laufer66} and Siu \cite{Siu67} to a Riemann domain over a Stein manifold.  When $\Omega$ is bounded, it follows from H\"{o}rmander's $L^2$-existence theorem for the $\dbar$-operator that if $\Omega$ is in addition pseudoconvex, then the $L^2$-cohomology groups $\widetilde H^{0, q}(\Omega)$ vanish for $q\ge 1$. The converse of H\"{o}rmander's theorem also holds, under the assumption that the interior of the closure of $\Omega$ is the domain itself.  Sheaf theoretic arguments for the Dolbeault cohomology groups can be modified to give a proof of this fact (cf.~\cite{Serre53, Laufer66, Siu67, Broecker83, Ohsawa88}; see also \cite{Fu05} and Section~\ref{sec:finite} below).

In this expository paper, we study positivity of the $\dbar$-Neumann Laplacian, in connection with the above-mentioned classical results, through the lens of spectral theory.  Our emphasis is on the interplay between spectral behavior of the $\dbar$-Neumann Laplacian and the geometry of the domains. This is evidently motivated by Marc Kac's famous question ``Can one hear the shape of a drum?" \cite{Kac66}.  Here we are interested in determining the geometry of a domain in $\C^n$ from the spectrum of the $\dbar$-Neumann Laplacian. (See \cite{Fu05, Fu08} for related results.) We make an effort to present a more accessible and self-contained treatment, using extensively spectral theoretic language but bypassing sheaf cohomology theory.

\section{Preliminaries}\label{sec:pre}

In this section, we review the spectral theoretic setup for the $\bar\partial$-Neumann
Laplacian. The emphasis here is slightly different from the one in the extant literature (cf.~\cite{FollandKohn72, ChenShaw99}).  The $\dbar$-Neumann Laplacian is defined through its associated quadratic form. As such, the self-adjoint
property and the domain of its square root come out directly from the definition.

Let $Q$ be a non-negative, densely defined, and closed
sesquilinear form on a complex Hilbert space $H$ with domain $\dom(Q)$. Then $Q$
uniquely determines a non-negative and
self-adjoint operator $S$ such that
$\dom(S^{1/2})=\dom(Q)$ and
\[
Q(u, v)=\langle S^{1/2}u, \; S^{1/2}v\rangle
\]
for all $u, v\in \dom(Q)$. (See Theorem 4.4.2 in \cite{Davies95}, to which we refer the reader
for the necessary spectral theoretic background used in this paper.)
For any subspace
$L\subset\dom(Q)$, let $\lambda(L)=\sup\{Q(u, u) \mid  u\in L,
\|u\|=1\}$.  For any positive integer $j$, let
\begin{equation}\label{minmax}
\lambda_{j} (Q)=\inf\{\lambda(L) \mid L\subset \dom(Q),
\dim(L)=j\}.
\end{equation}
The resolvent set $\rho(S)$ of $S$ consists of all $\lambda\in\C$ such that
the operator $S-\lambda I\colon\dom(S)\to H$ is both one-to-one and onto (and hence has
a bounded inverse by the closed graph theorem). The spectrum $\sigma(S)$, the complement of $\rho(S)$ in $\C$, is a non-empty closed subset of $[0, \ \infty)$.
Its bottom $\inf\sigma(S)$ is given by $\lambda_1(Q)$.  The essential spectrum $\sigma_e(S)$ is a closed subset of $\sigma(S)$ that consists of isolated eigenvalues
of infinite multiplicity and accumulation points of the spectrum.
It is empty if and only if $\lambda_{j}(Q)\to\infty$ as $j\to\infty$. In this case, $\lambda_{j}(Q)$ is the $j^{\text
{th}}$ eigenvalue of $S$, arranged in increasing order and repeated according to multiplicity.
The bottom of the essential spectrum $\inf\sigma_e(T)$ is the limit of $\lambda_j(Q)$ as $j\to\infty$. (When $\sigma_e(S)=\emptyset$, we set $\inf\sigma_e(S)=\infty$.)

Let $T_k\colon H_k\to
H_{k+1}$, $k=1, 2$, be densely defined and closed operators on
Hilbert spaces. Assume that $\range(T_1)\subset \kernel(T_2)$,
where $\range$ and $\kernel$ denote the range and kernel of the
operators. Let $T^*_k$ be the Hilbert space adjoint of $T_k$,
defined in the sense of Von Neumann by
\[
\dom(T^*_k)=\{u\in H_{k+1} \mid
\exists C>0, |\langle u, T_k v\rangle|\le C \|v\|,
\forall v\in \dom(T_k)\}
\]
and
\[
\langle T^*_k u, v\rangle=\langle u, T_k v\rangle, \quad \text{for all }
u\in\dom(T^*_k) \text{ and } v\in\dom(T_k).
\]
Then
$T^*_k$ is also densely defined and closed.  Let
\[
Q(u, v)=\langle T^*_1 u, T^*_1 v\rangle+\langle T_2 u, T_2 v\rangle
\]
with its domain given by $\dom(Q)=\dom(T^*_1)\cap\dom(T_2)$. The following proposition elucidates the above approach to the $\dbar$-Neumann Laplacian.

\begin{proposition}\label{prop:spectral1} $Q(u, v)$ is a densely defined, closed, non-negative sesquilinear form. The associated self-adjoint operator $\square$ is given by
\begin{equation}\label{eq:square}
\dom(\square)=\{f\in H_2 \mid f\in\dom(Q), T_2 f\in \dom(T_2^*), T_1^* f\in\dom(T_1)\},
\quad \square=T_1T_1^*+T_2^*T_2.
\end{equation}
\end{proposition}

\begin{proof} The closedness of $Q$ follows easily from that of $T_1$ and $T_2$. The non-negativity is evident. We now prove that $\dom(Q)$ is dense in $H_2$. Since
$\kernel(T_2)^\perp=\overline{\range(T_2^*)}\subset \kernel(T_1^*)$ and
\[
\dom(T_2)=\kernel (T_2)\oplus\big(\dom(T_2)\cap \kernel(T_2)^\perp\big),
\]
we have
\[
\dom(Q)=\dom(T_1^*)\cap\dom(T_2)=\big(\kernel (T_2)\cap\dom(T_1^*)\big)\oplus\big(\dom(T_2)\cap \kernel(T_2)^\perp\big).
\]
Since $\dom(T_1^*)$ and $\dom(T_2)$ are dense in $H_2$, $\dom(Q)$ is dense in $\kernel(T_2)\oplus
\kernel (T_2)^\perp=H_2$.

It follows from the above definition of $\square$ that $f\in\dom(\square)$ if and only if $f\in\dom(Q)$
and there exists a $g\in H_2$ such that
\begin{equation}\label{eq:q}
Q(u, f)=\langle u, g\rangle, \quad \text{for all } u\in\dom(Q)
\end{equation}
(cf.~Lemma 4.4.1 in \cite{Davies95}). Thus
\[
\dom(\square)\supset \{f\in H_2 \mid f\in\dom(Q), T_2 f\in \dom(T_2^*), T_1^* f\in\dom(T_1)\}.
\]
We now prove the opposite containment. Suppose $f\in\dom(\square)$. For any $u\in\dom(T_2)$, we write $u=u_1+u_2\in (\kernel(T_1^*)\cap\dom(T_2))\oplus \kernel(T_1^*)^\perp$. Note that $\kernel(T_1^*)^\perp \subset\range(T_2^*)^\perp=\kernel(T_2)$. It follows from
\eqref{eq:q} that
\[
|\langle T_2 u, T_2 f\rangle|=|\langle T_2 u_1, T_2 f\rangle=|Q(u_1, f)|=|\langle u_1, g\rangle|\le \|u\|\cdot\|g\|.
\]
Hence $T_2 f\in\dom(T_2^*)$. The proof of $T_1^* f\in\dom(T_1)$ is similar. For any $w\in\dom(T_1^*)$, we write $w=w_1+w_2\in (\kernel(T_2)\cap\dom(T_1^*))\oplus\kernel(T_2)^\perp$. Note that $\kernel(T_2)^\perp=\overline{\range(T_2^*)}\subset \kernel(T_1^*)$. Therefore, by \eqref{eq:q},
\[
|\langle T^*_1 w, T^*_1 f\rangle|=|\langle T_1^* w_1, T^*_1 f\rangle=|Q(w_1, f)|=|\langle w_1, g\rangle|\le \|w\|\cdot\|g\|.
\]
Hence $T_1^* f\in\dom(T_1^{**})=\dom(T_1)$. It follows from the definition of $\square$ that for any $f\in\dom(\square)$ and $u\in\dom(Q)$,
\[
\langle\square f, u\rangle=\langle \square^{1/2} f, \square^{1/2} u\rangle=Q(f, u)=\langle T_1^* f, T_1^*u\rangle + \langle T_2 f, T_2 u\rangle
=\langle (T_1T_1^*+T^*_2T_2)f, u\rangle.
\]
Hence $\square=T_1T_1^*+T^*_2T_2$. \end{proof}

The following proposition is well-known (compare
\cite{Hormander65}, Theorem 1.1.2 and Theorem 1.1.4; \cite{Catlin83}, Proposition 3;
and \cite{Shaw92}, Proposition 2.3). We provide a proof here for
completeness.

\begin{proposition}\label{prop:spectral2} $\inf\sigma(\square)>0$ if and only if
$\range(T_1)=\kernel(T_2)$ and $\range(T_2)$ is closed.
\end{proposition}

\begin{proof}
Assume $\inf\sigma(\square)>0$. Then $0$ is in the resolvent set of
$\square$ and hence $\square$ has a bounded inverse $G\colon H_2\to\dom(\square)$.
For any $u\in H_2$, write $u=T_1 T_1^*Gu +T_2^*T_2 Gu$. If $u\in
\kernel(T_2)$, then $0=(T_2u, \ T_2 Gu)=(T_2T_2^*T_2Gu,\ T_2Gu)
=(T_2^* T_2Gu, \ T_2^*T_2Gu)$.  Hence $T_2^*T_2Gu=0$ and
$u=T_1T_1^*Gu$. Therefore, $\range(T_1)=\kernel(T_2)$. Similarly,
$\range(T_2^*)=\kernel(T_1^*)$. Therefore $T_2^*$ and hence $T_2$ have closed
range. To prove the opposite implication,
we write $u=u_1+u_2 \in \kernel(T_2)\oplus\kernel(T_2)^\perp$, for any $u\in\dom(Q)$.
Note that $u_1, u_2\in\dom(Q)$. It follows
from $\kernel(T_2)=\range(T_1)$ and the closed range property of $T_2$
that there exists a positive constant $c$ such that $c\|u_1\|^2\le \|T_1^* u_1\|^2$
and $c\|u_2\|^2\le \|T_2 u_2\|^2$. Thus
\[
c\|u\|^2=c(\|u_1\|^2+\|u_2\|^2)\le
\|T_1^* u_1\|^2+\|T_2 u_2\|^2=Q(u, u).
\]
Hence $\inf\sigma (\square)\ge c>0$ (cf.~Theorem 4.3.1 in \cite{Davies95}).
\end{proof}

Let $\kernel(Q)=\kernel(T_1^*)\cap\kernel(T_2)$. Note that when it is non-trivial,
$\kernel(Q)$ is the eigenspace of the zero eigenvalue of $\square$. When $\range(T_1)$ is closed, $\kernel(T_2)=\range(T_1)\oplus\kernel(Q)$. For a subspace $L\subset H_2$, denote by $P_{L^\perp}$ the orthogonal projection onto $L^\perp$ and $T_2|_{L^\perp}$ the restriction of $T_2$ to $L^\perp$. The next proposition clarifies and strengthens the second part of Lemma 2.1 in \cite{Fu05}.

\begin{proposition}\label{prop:spectral3} The following statements are equivalent:
\begin{enumerate}
\item $\inf\sigma_e(\square)>0$.

\item $\range(T_1)$ and $\range(T_2)$ are closed and $\kernel(Q)$ is
finite dimensional.

\item There exists a finite dimensional subspace $L\subset \dom(T_1^*)\cap\kernel(T_2)$
such that $\kernel(T_2)\cap L^\perp =P_{L^\perp}(\range(T_1))$ and $\range(T_2|_{L^\perp})$ is closed.
\end{enumerate}
\end{proposition}

\begin{proof} We first prove (1) implies (2). Suppose $a=\inf\sigma_e (\square)>0$.  If $\inf\sigma(\square)>0$, then $\kernel(Q)$ is trivial and (2) follows from Proposition~\ref{prop:spectral2}.  Suppose $\inf\sigma(\square)=0$. Then $\sigma(\square)\cap [0, a)$ consists only of isolated points, all of which are eigenvalues of finite multiplicity of
$\square$ (cf. Theorem 4.5.2 in \cite{Davies95}).  Hence $\kernel(Q)$, the eigenspace of the eigenvalue $0$, is
finite dimensional.  Choose a sufficiently small $c>0$ so that $\sigma(\square)\cap [0, c)=\{0\}$. By the spectral theorem for self-adjoint operators (cf. Theorem 2.5.1 in \cite{Davies95}), there
exists a finite regular Borel measure $\mu$ on $\sigma(\square)\times\N$ and a unitary transformation $U\colon H_2\to L^2(\sigma(\square)\times \N, d\mu)$ such that $U\square U^{-1}=M_x$, where $M_x \varphi(x, n)=x \varphi(x, n)$ is the multiplication operator by $x$ on
$L^2(\sigma(\square)\times\N, d\mu)$.  Let $P_{\kernel(Q)}$ be the orthogonal projection onto $\kernel(Q)$. For any $f\in\dom(Q)\cap \kernel(Q)^\perp$,
\[
UP_{\kernel(Q)} f= \chi_{[0, c)} Uf =0,
\]
where $\chi_{[0, c)}$ is the characteristic function of $[0, c)$. Hence $Uf$
is supported on $[c, \infty)$.  Therefore,
\[
Q(f, f)=\int_{\sigma(\square)\times\N} x |Uf|^2\, d\mu\ge c\|Uf\|^2=c\|f\|^2.
\]
It then follows from Theorem 1.1.2 in \cite{Hormander65} that both $T_1$ and $T_2$ have
closed range.

To prove (2) implies (1), we use Theorem 1.1.2 in \cite{Hormander65} in
the opposite direction: There exists a positive constant $c$ such that
\begin{equation}\label{eq:h}
c\|f\|^2\le Q(f, f), \quad \text{for all } f\in\dom(Q)\cap\kernel(Q)^\perp.
\end{equation}
Proving by contradiction, we assume $\inf\sigma_e(\square)=0$. Let $\eps$ be
any positive number less than $c$. Since $L_{[0, \eps)}=\range(\chi_{[0, \eps)}(\square))$
is infinite dimensional (cf. Lemma 4.1.4 in \cite{Davies95}), there exists a non-zero $g\in L_{[0, \eps)}$
such that $g\perp \kernel(Q)$.  However,
\[
Q(g, g)=\int_{\sigma(\square)\times\N} x\chi_{[0, \eps)}|Ug|^2 \, d\mu \le \eps \|Ug\|^2=\eps\|g\|^2,
\]
contradicting \eqref{eq:h}.

We do some preparations before proving the equivalence of (3) with (1) and (2).
Let $L$ be any finite dimensional subspace of $\dom(T_1^*)\cap\kernel(T_2)$.
Let $H_2'=H_2\ominus L$. Let
$T_2'=T_2\big|_{H_2'}$ and let ${T_1^*}'=T_1^*\big|_{H_2'}$.  Then
$T_2'\colon H_2'\to H_3$ and ${T_1^*}'\colon H_2'\to H_1$ are
densely defined, closed operators. Let $T_1'\colon H_1\to H_2'$ be
the adjoint of ${T_1^*}'$.  It follows from the definitions that
$\dom(T_1)\subset \dom(T_1')$.  The finite dimensionality of $L$
implies the opposite containment. Thus, $\dom(T_1)=\dom(T_1')$.
For any $f\in\dom(T_1)$ and $g\in\dom({T_1^*}')=\dom(T_1^*)\cap L^\perp$,
\[
\langle T'_1 f, g\rangle =\langle f, {T_1^*}'g\rangle=\langle f, T_1^* g\rangle
=\langle T_1 f, g\rangle.
\]
Hence $T_1'=P_{L^\perp}\circ T_1$ and $\range(T_1')=P_{L^\perp}(\range(T_1))\subset \kernel(T_2')$.  Let
\[
Q'(f, g)=\langle {T_1'}^* f, {T_1'}^*g\rangle+\langle T_2' f, T_2'g\rangle
\]
be the associated sesquilinear form on $H_2'$ with $\dom(Q')=\dom(Q)\cap L^\perp$.

We are now in position to prove that (2) implies (3). In this case, we can take
$L=\kernel(Q)$.  By Theorem 1.1.2 in \cite{Hormander65}, there exists a
positive constant $c$ such that
\[
Q(f, f)=Q'(f, f)\ge c\|f\|^2, \quad \text{for all } f\in\dom(Q').
\]
We then obtain (3) by applying Proposition~\ref{prop:spectral2} to
$T_1'$, $T_2'$, and $Q'(f, g)$.

Finally, we prove (3) implies (1). Applying Proposition~\ref{prop:spectral2} in
the opposite direction, we know that there exists a positive constant $c$ such that
\[
Q(f, f)\ge c\|f\|^2, \quad\text{for all } f\in\dom(Q)\cap L^\perp.
\]
The rest of the proof follows the same lines of the above proof of the implication $(2)\Rightarrow (1)$, with $\kernel(Q)$ there replaced by $L$.
\end{proof}

We now recall the definition of the $\dbar$-Neumann Laplacian
on a complex manifold. Let $X$ be a complex hermitian manifold of dimension $n$.
Let $C^\infty_{(0, q)}(X)=C^\infty(X, \Lambda^{0, q}T^*X)$ be
the space of smooth $(0, q)$-forms on $X$.
Let $\dbar_q\colon C^\infty_{(0, q)}(X)\to C^\infty_{(0, q+1)}(X)$
be the composition of the exterior differential
operator and the projection onto $C^\infty_{(0, q+1)}(X)$.

Let $\Omega$ be a domain in $X$. For $u, v\in C^\infty_{(0, q)}(X)$,
let $\langle u, v\rangle$ be the
point-wise inner product of $u$ and $v$, and let
\[
\langle\langle u, v\rangle\rangle_\Omega=\int_\Omega \langle u,
v\rangle dV
\]
be the inner product of $u$ and $v$ over $\Omega$.  Let $L^2_{(0,
q)}(\Omega)$ be the completion of the space of compactly supported forms in
$C^\infty_{(0, q)}(\Omega)$ with respect to the above inner
product.  The operator $\dbar_q$ has a closed extension on
$L^2_{(0, q)}(\Omega)$. We also denote the closure by $\dbar_q$.
Thus $\dbar_q\colon L^2_{(0, q)}(\Omega)\to L^2_{(0, q+1)}(\Omega)$
is densely defined and closed. Let $\dbarstar_q$ be its adjoint.
For $1\le q\le n-1$, let
\[
Q_q(u, v)=\langle\langle\dbar_q u, \dbar_q
v\rangle\rangle_\Omega+\langle\langle\dbarstar_{q-1} u,
\dbarstar_{q-1} v\rangle\rangle_\Omega
\]
be the sesquilinear form on $L^2_{(0, q)}(\Omega)$ with domain
$\dom(Q_{q})=\dom(\dbar_q)\cap \dom(\dbarstar_{q-1})$.
The self-adjoint operator $\square_{q}$ associated with
$Q_{q}$ is  called {\it the $\dbar$-Neumann Laplacian} on $L^2_{(0, q)}(\Omega)$.  It is an
elliptic operator with non-coercive boundary conditions \cite{KohnNirenberg65}.

The Dolbeault and $L^2$-cohomology groups on $\Omega$
are defined respectively by
\[
H^{0, q}(\Omega)=\frac{\{f\in C^\infty_{(0, q)}(\Omega)\mid
\dbar_q f=0\}}{\{\dbar_{q-1} g \mid g\in C^\infty_{(0, q-1)}(\Omega)\}} \ \text{ and }\
\widetilde H^{0, q}(\Omega)=\frac{\{f\in L^2_{(0,
q)}(\Omega)\mid \dbar_q f=0\}}{\{\dbar_{q-1} g \mid g\in L^2_{(0, q-1)}(\Omega)\}}.
\]
These cohomology groups are in general not isomorphic. For example, when a complex variety
is deleted from $\Omega$, the $L^2$-cohomology group remains the same but the Dolbeault cohomology group could change from trivial to infinite dimensional.
As noted in the paragraph preceding Proposition~\ref{prop:spectral3}, when $\range(\dbar_{q-1})$ is closed in  $L^2_{(0, q)}(\Omega)$, $\widetilde H^{0, q}(\Omega)\cong \kernel(\square_q)$.
We refer the reader to \cite{Demailly} for an extensive treatise on the subject and to \cite{Hormander65} and \cite{Ohsawa82} for results relating these cohomology groups.

\section{Positivity of the spectrum and essential spectrum}\label{sec:finite}

Laufer proved in \cite{Laufer75} that for any open subset of a Stein manifold, if a Dolbeault cohomology group is finite dimensional, then it is trivial. In this section, we establish the
following $L^2$-analogue of this result on a bounded domain in a Stein mainfold:

\begin{theorem}\label{th:finite} Let $\Omega\subset\subset X$ be a domain in a Stein manifold $X$ with $C^1$ boundary. Let $\square_q$, $1\le q\le n-1$, be the $\dbar$-Neumann Laplacian on $L^2_{(0, q)}(\Omega)$. Assume that $\scriptn(\square_q)\subset W^1(\Omega)$. Then $\inf\sigma(\square_q)>0$ if and only if $\inf\sigma_e(\square_q)>0$.
\end{theorem}

The proof of Theorem~\ref{th:finite} follows the same line of arguments as Laufer's. We provide the details below.

Let $H^\infty(\Omega)$ be the space of bounded holomorphic functions on $\Omega$. For any $f\in H^\infty(\Omega)$, let $M_f$ be the multiplication operator by $f$:
\[
M_f\colon L^2_{(0, q)}(\Omega)\to L^2_{(0, q)}(\Omega), \qquad M_f(u)=fu.
\]
Then $M_f$ induces an endomorphism on $\widetilde H^{0, q}(\Omega)$.  Let $\scripti$ be
set of all holomorphic functions $f\in H^\infty(\Omega)$ such that $M_f=0$ on $\widetilde H^{0, q}(\Omega)$.  Evidently, $\scripti$ is an ideal of $H^\infty(\Omega)$. Assume $\inf\sigma_e(\square_q)>0$. To show that $\widetilde H^{0, q}(\Omega)$ is trivial, it suffices to show that $1\in\scripti$.

\begin{lemma}\label{lm:vector} Let $\xi$ be a holomorphic vector field on $X$ and let $f\in\scripti$. Then
$\xi(f)\in \scripti$.
\end{lemma}

\begin{proof} Let $D=\xi\lrcorner \partial\colon C^\infty_{(0, q)}(\Omega)\to
C^\infty_{(0, q)}(\Omega)$, where $\lrcorner$ denotes the contraction operator.
It is easy to check that $D$ commute with the
$\dbar$ operator. Therefore, $D$ induces an endomorphism on $\widetilde H^{0, q}(\Omega)$.
(Recall that under
the assumption, $\widetilde H^{0, q}(\Omega)\cong \kernel(\square_q)\subset W^1(\Omega)$.) For any $u\in \kernel(\square_q)$,
\[
D(fu)-fD(u)=\xi\lrcorner\partial(fu)-f\xi\lrcorner\partial u=\xi(f) u.
\]
Notice that $\Omega$ is locally starlike near the boundary. Using partition of unity and the Friedrichs Lemma, we obtain $[D(fu)]=0$. Therefore, $\left[\xi(f) u\right]=
\left[D(fu)\right]-\left[fD(u)\right]=\left[0\right]$. \end{proof}

We now return to the proof of the theorem. Let $F=(f_1, \ldots, f_{n+1})\colon X\to \C^{2n+1}$ be a proper embedding of $X$ into $\C^{2n+1}$ (cf.~Theorem 5.3.9 in \cite{Hormander91}). Since $\Omega$ is relatively compact in $X$, $f_j\in H^\infty(\Omega)$. For any $f_j$,  let $P_j(\lambda)$ be the characteristic polynomial of $M_{f_j}\colon \widetilde H^{0, q}(\Omega)\to \widetilde H^{0, q}(\Omega)$. By the Cayley-Hamilton theorem, $P_j(M_{f_j})=0$ (cf.~Theorem~2.4.2 in \cite{HornJohnson}). Thus $P_j(f_j)\in \scripti$.

The number of points in the set $\{(\lambda_1, \lambda_2, \ldots, \lambda_{2n+1})\in\C^{2n+1} \mid P_j(\lambda_j)=0, 1\le j\le 2n+1 \}$ is finite.  Since $F\colon X\to \C^{2n+1}$ is one-to-one, the number of common zeroes of $P_j(f_j(z))$, $1\le j\le 2n+1$, on $X$ is also finite. Denote these zeroes by $z^k$, $1\le k\le N$. For each  $z^k$, let $g_k$ be a function in $\scripti$ whose vanishing order at $z^k$ is minimal. (Since $P_j(f_j)\in\scripti$, $g_k\not\equiv0$.) We claim that $g_k(z^k)\not=0$.  Suppose otherwise $g_k(z^k)=0$.  Since there exists a holomorphic vector field $\xi$ on $X$ with any prescribed holomorphic tangent vector at any given point (cf.~Corollary 5.6.3 in \cite{Hormander91}), one can find an appropriate choice of
$\xi$ so that $\xi(g_j)$ vanishes to lower order at $z^k$. According to Lemma~\ref{lm:vector}, $\xi(g_j)\in\scripti$. We thus arrive at a contradiction.

Now we know that there are holomorphic functions, $P_j(f_j)$, $1\le j\le 2n+1$, and  $g_k$, $1\le k\le N$, that have no common zeroes on $X$. It then follows that there exist holomorphic functions
$h_j$ on $X$ such that
\[
\sum P_j(f_j) h_j +\sum g_k h_k =1.
\]
(See, for example, Corollary 16 on p.~244 in \cite{GunningRossi65}, Theorem 7.2.9 in \cite{Hormander91}, and Theorem~7.2.5 in \cite{Krantz01}. Compare also Theorem 2 in \cite{Skoda72}.) Since $P_j(f_j)\in\scripti, g_k\in \scripti$, and $h_j\in H^\infty(\Omega)$, we have $1\in \scripti$.  We thus conclude the proof of Theorem~\ref{th:finite}.

\begin{remark}
(1) Unlike the above-mentioned result of Laufer on the Dolbeault cohomology groups \cite{Laufer75}, Theorem~\ref{th:finite} is not expected to hold if the boundedness condition on $\Omega$ is removed (compare \cite{Wiegerinck84}). It would be interesting to know whether
Theorem~\ref{th:finite} remains true if the assumption $\scriptn(\square_q)\subset W^1(\Omega)$ is dropped and whether it remains true for unbounded pseudoconvex domains.

(2) Notice that in the above proof, we use the fact that $\range(\dbar_{q-1})$ is closed, as a consequence of the assumption $\inf\sigma_e(\square_q)>0$ by Proposition~\ref{prop:spectral3}. It is well known that for any infinite dimensional Hilbert space $H$, there exists a subspace $R$ of $H$ such that $H/R$ is finite dimensional but $R$ is not closed. However, the construction of such a subspace usually involves Zorn's lemma (equivalently, the axiom of choice). It would be of interest to know whether there exists a domain $\Omega$ in a Stein
 manifold such that $\widetilde H^{0, q}(\Omega)$ is finite dimensional but $\range(\dbar_{q-1})$ is not closed.

 (3) We refer the reader to \cite{Shaw09} for related results on the relationship between triviality and finite dimensionality of the $L^2$-cohomology gourps using the $\dbar$-Cauchy problem. We also refer the reader to \cite{Brinkschulte02} for a related result on embedded $CR$ manifolds.
\end{remark}

\section{Hearing pseudoconvexity}\label{sec:psc}

The following theorem illustrates that one can easily determine pseudoconvexity from
the spectrum of the $\dbar$-Neumann Laplacian.

\begin{theorem}\label{th:psc} Let $\Omega$ be a bounded domain in $\C^n$ such that
$\interior(\cl(\Omega))=\Omega$. Then the following statements are equivalent:
\begin{enumerate}
\item $\Omega$ is pseudoconvex.
\item $\inf\sigma(\square_q)>0$, for all $1\le q\le n-1$.
\item $\inf\sigma_e(\square_q)>0$, for all $1\le q\le n-1$.
\end{enumerate}
\end{theorem}

The implication $(1)\Rightarrow (2)$ is a consequence of H\"{o}rmander's fundamental
$L^2$-estimates of the $\dbar$-operator \cite{Hormander65}, in light of Proposition~\ref{prop:spectral2}, and it holds without the assumption $\interior(\cl(\Omega))=\Omega$. The implications $(2)\Rightarrow (1)$ and $(3)\Rightarrow (1)$ are consequences of the sheaf cohomology theory dated back to Oka and Cartan (cf.~\cite{Serre53, Laufer66, Siu67, Broecker83, Ohsawa88}).  A elementary proof of (2) implying (1), as explained in \cite{Fu05}, is given below. The proof uses sheaf cohomology arguments in \cite{Laufer66}. When adapting Laufer's method to study the $L^2$-cohomology groups, one encounters a difficulty: While the restriction to the complex hyperplane of the smooth function resulting from the sheaf cohomology arguements for the Dolbeault cohomology groups is well-defined, the restriction of the corresponding $L^2$ function is not. This difficulty
was overcome in \cite{Fu05} by appropriately modifying the construction of auxiliary $(0, q)$-forms (see the remark after the proof for more elaborations on this point).

We now show that (2) implies (1). Proving by contradiction, we assume that $\Omega$ is not pseudoconvex.
Then there exists a domain $\widetilde\Omega\supsetneqq\Omega$ such that every holomorphic
function on $\Omega$ extends to $\widetilde\Omega$. Since $\interior(\cl(\Omega))=\Omega$,
$\widetilde\Omega\setminus\cl(\Omega)$ is non-empty. After a
translation and a unitary transformation, we may assume that the
origin is in $\widetilde\Omega\setminus\cl(\Omega)$ and there is
a point $z^0$ in the intersection of
$z_n$-plane with $\Omega$ that is in the same connected component
of the intersection of the $z_n$-plane with $\widetilde\Omega$.

Let $m$ be a positive integer (to be specified later). Let $k_q=n$. For any
$\{k_1, \ldots, k_{q-1}\}\subset \{1, 2, \ldots, n-1\}$, we define
\begin{equation}\label{eq:rm}
u(k_1, \ldots, k_q)=\frac{(q-1)!
(\bar z_{k_1}\cdots\bar z_{k_q})^{m-1}}{ r_m^q}\sum_{j=1}^q (-1)^j \bar{z}_{k_j} d\bar{z}_{k_1}\wedge
\ldots\wedge\widehat{d\bar{z}_{k_j}}\wedge\ldots\wedge
d\bar{z}_{k_q},
\end{equation}
where $r_m=|z_1|^{2m}+\ldots +|z_n|^{2m}$. As usual,
$\widehat{d\bar{z}_{k_j}}$ indicates the deletion of
$d\bar z_{k_j}$ from the wedge product. Evidently,
$u(k_1, \ldots, k_q)\in L^2_{(0, q-1)}(\Omega)$ is a smooth form on
$\C^n\setminus\{ 0\}$.
Moreover, $u(k_1, \ldots, k_q)$ is skew-symmetric with respect to the
indices $(k_1, \ldots, k_{q-1})$.  In particular, $u(k_1, \ldots, k_q)=0$ when two $k_j$'s are identical.

We now fix some notional conventions. Let $K=(k_1,\ldots, k_q)$ and
$J$ a collection of indices from $\{k_1, \ldots, k_q\}$.  Write $d\bar z_K=d\bar z_{k_1}\wedge \ldots\wedge
d\bar z_{k_q}$, $\bar z_K^{m-1}= (\bar z_{k_1}\cdots\bar
z_{k_q})^{m-1}$, and $\widetilde{d\bar
z_{k_j}}=d\bar{z}_{k_1}\wedge
\ldots\wedge\widehat{d\bar{z}_{k_j}}\wedge\ldots\wedge
d\bar{z}_{k_q}$.  Denote by
$(k_1, \ldots, k_q \mid J)$ the tuple of remaining indices after deleting those
in $J$ from $(k_1, \ldots, k_q)$. For example, $(2, 5, 3, 1 \mid (4, 1, 6 \mid 4, 6))=(2, 5, 3)$.

It follows from a straightforward calculation that
\begin{align}
\dbar u(k_1, \ldots, k_q)&=-\frac{q! m
\bar{z}_K^{m-1}} {r_m^{q+1}}\big(r_m d\bar z_K
+ \big(\sum_{\ell =1}^n \bar z_\ell^{m-1} z_\ell^m d\bar
z_\ell\big)\wedge \big(\sum_{j=1}^q (-1)^j \bar
z_{k_j}\widetilde{d\bar z_{k_j}}\big)\big)\notag\\
&=-\frac{q! m\bar
z_K^{m-1}}{r_m^{q+1}} \sum_{\begin{subarray}{c}\ell=1 \\
\ell\not=k_1, \ldots, k_q\end{subarray}}^n z^m_\ell \bar
z_\ell^{m-1}\big(\bar z_\ell d\bar z_K+d\bar
z_\ell\wedge\sum_{j=1}^q(-1)^j \bar z_{k_j}
\widetilde{d\bar z_{k_j}}\big)\notag\\
&=m\sum_{\ell=1}^{n-1} z^m_\ell u(\ell, k_1, \ldots,
k_q).\label{eq:um}
\end{align}
It follows that $u(1, \ldots, n)$ is a $\dbar$-closed $(0, n-1)$-form.

By Proposition~\ref{prop:spectral2}, we have $\range(\dbar_{q-1})=\kernel(\dbar_q)$
for all $1\le q\le n-1$. We now solve
the $\dbar$-equations inductively, using $u(1, \ldots, n)$ as initial
data. Let $v\in L^2_{(0, n-2)}(\Omega)$ be a solution to
$\dbar v=u (1, \ldots, n)$.  For any $k_1\in \{1, \ldots, n-1\}$, define
\[
w(k_1)=-mz^m_{k_1} v+(-1)^{1+k_1} u(1, \ldots, n \mid k_1).
\]
Then it follows from \eqref{eq:um} that $\dbar w(k_1)=0$.  Let
$v(k_1)\in L^2_{(0, n-3)}(\Omega)$ be a solution of $\dbar v(k_1)=w(k_1)$.

Suppose for any $(q-1)$-tuple $K'=(k_1, \ldots, k_{q-1})$ of integers from $\{1, \ldots, n-1\}$, $q\ge 2$, we have constructed $v(K')\in L^2_{(0, n-q-1)}(\Omega)$ such that it is skew-symmetric with respect to the indices and satisfies
\begin{equation}\label{eq:v}
\dbar v(K')=m\sum_{j=1}^{q-1} (-1)^j z^m_{k_j} v(K' \mid
{k_j}) + (-1)^{q+|K'|} u(1, \ldots, n \mid {K'})
\end{equation}
where $|K'|=k_1+ \ldots +k_{q-1}$ as usual.  We now construct a $(0, n-q-2)$-forms $v(K)$
satisfying \eqref{eq:v} for any $q$-tuple $K=(k_1, \ldots, k_q)$ of integers from $\{1, \ldots, n-1\}$ (with $K'$ replaced by $K$).  Let
\[
w(K)=m\sum_{j=1}^q (-1)^j z^m_{k_j} v(K \mid {k_j})
+(-1)^{q+|K|} u(1, \ldots, n \mid  K) .
\]
Then it follows from \eqref{eq:um} that
\begin{align*}
\dbar w (K)&=m\sum_{j=1}^q (-1)^j z^m_{k_j}\dbar v(K \mid
{k_j})
+(-1)^{q+|K|} \dbar u(1, \ldots, n \mid {K})\\
&=m\sum_{j=1}^q (-1)^j z^m_{k_j} \Big( m\sum_{1\le i <j} (-1)^i
z^m_{k_i} v(K \mid  k_j,  k_i) +m\sum_{j<i\le q} (-1)^{i-1}
z^m_{k_i}  v(K \mid  k_j,
 k_i) \\
&\qquad -(-1)^{q+|K|-k_j} u(1, \ldots, n \mid {(K \mid  k_j)})
\Big) +(-1)^{q +|K|} \dbar u(1, \ldots, n \mid  K) \\
&=(-1)^{q+|K|}\big( -m\sum_{j=1}^q (-1)^{j-k_j} z^m_{k_j} u(1,
\ldots, n \mid {(K \mid  k_j)}) +\dbar
u(1, \ldots, n \mid  K)\big) \\
&=(-1)^{q+|K|}\big( -m\sum_{j=1}^q z^m_{k_j} u(k_j, (1, \ldots, n \mid
 K)) + \dbar u(1, \ldots, n \mid  K)\big)=0
\end{align*}
Therefore, by the hypothesis, there exists a $v(K)\in L^2_{(0, n-q-2)}(\Omega)$ such that
$\dbar v(K)=w(K)$. Since $w(K)$ is skew-symmetric with respect to indices $K$, we
may also choose a likewise $v(K)$. This then concludes the inductive step.

Now let
\[
F=w(1, \ldots, n-1)= m\sum_{j=1}^{n-1} z^m_j v(1, \ldots,
j, \ldots, n-1)- (-1)^{n+\frac{n(n-1)}2} u(n),
\]
where $u(n)=-\bar{z}_n^m/r_m$, as given by \eqref{eq:rm}.
Then $F(z)\in L^2(\Omega)$ and $\dbar F(z) =0$.  By the hypothesis, $F(z)$ has a
holomorphic extension to $\widetilde{\Omega}$.  We now restrict $F(z)$ to
the coordinate hyperplane $z'=(z_1, \ldots, z_{n-1})=0$. Notice that
so far we only choose the $v(K)$'s and $w(K)$'s from  $L^2$-spaces. The
restriction to the coordinate hyperplane $z'=0$ is not well-defined.
To overcome this difficulty, we choose $m>2(n-1)$. For sufficiently
small $\e>0$ and $\delta>0$,
\begin{align*}
&\left\{\int_{\{|z'|<\e\}\cap\Omega}
\big|\big(F+(-1)^{n+\frac{n(n-1)}2}u(n)\big)(\delta z', z_n))\big|^2dV(z)\right\}^{1/2}\\
&\qquad\le m \delta^m\e^m \sum_{j=1}^{n-1}\left\{\int_{\{|z'|<\e\}\cap\Omega}
|v(1, \ldots, \hat j, \ldots, n-1)(\delta z', z_n)|^2
dV(z)\right\}^{1/2} \\
&\qquad\le m\delta^{m-2(n-1)}\e^m\sum_{j=1}^{n-1} \|v(1, \ldots,
\hat j, \ldots, n-1)\|_{L^2(\Omega)}.
\end{align*}
Letting $\delta\to 0$, we then obtain
\[
F(0, z_n)=-(-1)^{n+\frac{n(n-1)}2} u(n)(0, z_n)=(-1)^{n+\frac{n(n-1)}2}z_n^{-m}.
\]
for $z_n$ near $z^0_n$.  (Recall that $z^0\in\Omega$ is in the same
connected component of $\{z'=0\}\cap\widetilde\Omega$ as the origin.)
This contradicts the analyticity of $F$ near the origin. We
therefore conclude the proof of Theorem~\ref{th:psc}.

\begin{remark} (1) The above proof of the implication $(2)\Rightarrow (1)$ uses only the fact that
the $L^2$-cohomology groups $\widetilde H^{0, q}(\Omega)$ are trivial for all $1\le q\le n-1$. Under the (possibly) stronger assumption $\inf\sigma(\square_q)>0$, $1\le q\le n-1$, the difficulty regarding the restriction of the $L^2$ function to the complex hyperplane in the proof becomes superficial. In this case, the $\dbar$-Neumannn Laplacian $\square_q$ has a bounded inverse. The interior ellipticity of the $\dbar$-complex implies that one can in fact choose
the forms $v(K)$ and $w(K)$ to be smooth inside $\Omega$, using the canonical solution operator to the $\dbar$-equation. Therefore, in this case, the restriction to $\{z'=0\}\cap\Omega$ is well-defined.  Hence one can choose $m=1$. This was indeed the choice in \cite{Laufer66}, where the forms involved are smooth and the restriction posts no problem.  It is interesting to note that by having the freedom to choose $m$ sufficiently large, one can leave out the use of interior ellipticity. Also, the freedom to choose $m$ becomes
crucial when one proves an analogue of Theorem~\ref{th:psc} for the Kohn Laplacian because the $\dbar_b$-complex is no longer elliptic. The construction of $u(k_1, \ldots, k_q)$ in
\eqref{eq:rm} with the exponent $m$ was introduced in \cite{Fu05} to handle this
difficulty.

(2) One can similarly give a proof of the implication $(3)\Rightarrow (1)$. Indeed, the above proof can be easily modified to show that the finite dimensionality of $\widetilde H^{0, q}(\Omega)$, $1\le q\le n-1$, implies the pseudoconvexity of $\Omega$. In this case, the $u(K)$'s are defined by
\[
u(k_1, \ldots, k_q)=\frac{(\alpha+q-1)! \bar
z_n^{m\alpha}(\bar z_{k_1}\cdots\bar z_{k_q})^{m-1}}{ r_m^{\alpha+
q}}\sum_{j=1}^q (-1)^j \bar{z}_{k_j} d\bar{z}_{k_1}\wedge
\ldots\wedge\widehat{d\bar{z}_{k_j}}\wedge\ldots\wedge
d\bar{z}_{k_q},
\]
where $\alpha$ is any non-negative integers. One now
fixes a choice of $m>2(n-1)$ and let $\alpha$ runs from $0$ to $N$ for a sufficiently
large $N$, depending on the dimensions of the $L^2$-cohomology groups. We refer the reader to \cite{Fu05} for details.

(3) As noted in Sections~\ref{sec:pre} and \ref{sec:finite}, unlike the Dolbeault cohomology case, one cannot remove the assumption $\interior(\cl(\Omega))=\Omega$ or the boundedness
condition on $\Omega$ from Theorem~\ref{th:psc}.  For example, a
bounded pseudoconvex domain in $\C^n$ with a complex analytic variety removed still satisfies condition (2) in Theorem~\ref{th:finite}.

(4) As in \cite{Laufer66}, Theorem~\ref{th:psc} remains true for a Stein manifold. More generally, as a consequence of Andreotti-Grauert's theory \cite{AndreottiGrauert62}, the $q$-convexity of a bounded domain $\Omega$ in a Stein manifold such that $\interior(\cl(\Omega))=\Omega$ is characterized by $\inf\sigma(\square_k)>0$ or $\inf\sigma_e(\square_k)>0$ for all $q\le k\le n-1$.

(5) It follows from Theorem~3.1 in \cite{Hormander04} that for a domain $\Omega$ in a complex hermitian manifold of dimension $n$, if $\inf\sigma_e(\square_q)>0$ for some $q$ between $1$ and $n-1$, then wherever the boundary is $C^3$-smooth, its Levi-form cannot have exactly $n-q-1$ positive and $q$ negative eigenvalues.  A complete characterization of a domain in a complex hermitian manifold, in fact, even in $\C^n$, that has $\inf\sigma_e(\square_q)>0$ or $\inf\sigma(\square_q)>0$ is unknown.

\end{remark}

\bigskip

\noindent{\bf Acknowledgments.} We thank Professor Yum-Tong Siu for drawing our attention to the work of Laufer \cite{Laufer66}, by which our work was inspired. We also thank Professors Mei-Chi Shaw and Emil Straube, and the referee for stimulating discussions and constructive suggestions.

\bibliography{survey}

\begin{thebibliography}{XXXX}

\bibitem[AG62]{AndreottiGrauert62}
Aldo Andreotti and Hans Grauert, \emph{Th\'{e}or\`{e}me de finitude pour la cohomologie des espaces complexes},  Bull. Soc. Math. France  \textbf{90} (1962), 193-259.

\bibitem[BSt99]{BoasStraube99}
Harold~P. Boas and Emil~J. Straube,
\emph{Global regularity of the $\dbar$-Neumann problem:
a survey of the $L^2$-{S}obolev theory}, Several Complex Variables
(M. Schneider and Y.-T. Siu, eds.), MSRI Publications, vol.~37, 79-112,
1999.

\bibitem[B02]{Brinkschulte02}
Judith Brinkschulte, \emph{Laufer's vanishing theorem for embedded CR manifolds}, Math. Z. \textbf{239}(2002), 863--866.

\bibitem[Br83]{Broecker83}
Thorsten Broecker, \emph{Zur $L\sp{2}$-Kohomologie beschr\"{a}nkter Gebiete},
Bonner Mathematische Schriften, vol. 145, Universität Bonn, 1983.


\bibitem[C83]{Catlin83}
David Catlin,
\emph{Necessary conditions for subellipticity of the
  {$\overline\partial$}-{Neumann} problem}, Ann. Math.
  \textbf{117} (1983), 147--171.


\bibitem[CS99]{ChenShaw99}
So-Chin Chen and Mei-Chi Shaw,
\emph{Partial differential equations in several complex variables},
AMS/IP, 2000.

\bibitem[Ch99]{Christ99}
Michael Christ, {\em Remarks on global irregularity in the
$\overline\partial$-Neumann problem}, Several Complex Variables
(M. Schneider and Y.-T. Siu, eds.), MSRI Publications, vol.~37,
161-198, 1999.

\bibitem[DK99]{DangeloKohn99}
J.~D'Angelo and J.~J.~Kohn, {\em   Subelliptic estimates and
finite type}, Several Complex Variables (M. Schneider and Y.-T.
Siu, eds.), MSRI Publications, vol.~37, 199-232, 1999.

\bibitem[D95]{Davies95}
E.~B.~Davies, \emph{Spectral theory and differential operators},
Cambridge Studies in advanced mathematics, vol.~42, Cambridge
University Press, 1995.

\bibitem[De]{Demailly}
J.-P. Demailly, \emph{Complex analytic and algebraic
geometry}, book available at \\ \text{
http://www-fourier.ujf-grenoble.fr/$\widetilde{\;\;}$demailly/books.html}.

\bibitem[FK72]{FollandKohn72}
G.~B. Folland and J.~J. Kohn, \emph{The {Neumann} problem for the
  {Cauchy-Riemann} complex}, Annals of Mathematics Studies, no.~75, Princeton University Press, 1972.

\bibitem[Fu05]{Fu05}
Siqi Fu, \emph{Hearing pseudoconvexity with the Kohn Laplacian}, Math. Ann. \textbf{331} (2005), 475-485.

\bibitem[Fu08]{Fu08}
\bysame, \emph{Hearing the type of a domain in C2 with the d-bar-Neumann Laplacian}, Adv. in Math. \textbf{219} (2008), 568-603.

\bibitem[FS01]{FuStraube01}
Siqi Fu and Emil J.~Straube, \emph{Compactness in the
$\dbar$-{N}eumann problem}, Complex Analysis and Geometry,
Proceedings of Ohio State University Conference, vol.~9, 141-160,
Walter De Gruyter, 2001.

\bibitem[GR65]{GunningRossi65}
Robert C.~Gunning and Hugo Rossi, \emph{Analytic functions of several complex variables}, Prentice-Hall, Englewood Cliffs, New Jersey, 1965.

\bibitem[H65]{Hormander65}
Lars H{\"{o}}rmander, \emph{{$L\sp{2}$} estimates and existence
theorems for the {$\overline\partial$} operator}, Acta Math.
\textbf{113} (1965), 89--152.

\bibitem[H91]{Hormander91}
\bysame, \emph{An introduction to complex analysis in several
variables}, third ed., Elsevier Science Publishing, 1991.

\bibitem[H04]{Hormander04}
\bysame, \emph{The null space of the $\bar\partial$-Neumann operator}, Ann. Inst. Fourier (Grenoble) \textbf{54} (2004), 1305-1369.

\bibitem[HJ85]{HornJohnson}
Roger A. Horn and Charles R. Johnson, \emph{Matrix analysis}, Cambridge University Press, 1985.

\bibitem[Ka66]{Kac66}
Marc Kac,
\emph{Can one hear the shape of a drum?}
Amer. Math. Monthly \textbf{73} (1966), 1--23.

\bibitem[Ko63]{Kohn63}
J. J. Kohn, \emph{Harmonic integrals on strongly pseudo-convex
manifolds,~{I}},  Ann. Math. \textbf{78} (1963), 112--148.

\bibitem[Ko64]{Kohn64}
\bysame, \emph{Harmonic integrals on strongly pseudo-convex
manifolds,~{II}}, Ann. Math. \textbf{79} (1964), 450--472.

\bibitem[KN65]{KohnNirenberg65}
J.~J.~Kohn and L.~Nirenberg, \emph{Non-coercive boundary value problems}, Comm. Pure Appl. Math. \textbf{18} (1965), 443-492.

\bibitem[Kr01]{Krantz01}
Steven G. Krantz, \emph{Function theory of several complex
variables}, AMS Chelsea Publishing, Providence, RI, 2001.

\bibitem[L66]{Laufer66}
Henry B. Laufer, \emph{On sheaf cohomology and envelopes of
holomorphy}, Ann. Math. \textbf{84} (1966), 102-118.

\bibitem[L75]{Laufer75}
\bysame, \emph{On the finite dimensionality of the Dolbeault cohomology groups}, Proc. Amer. Math. Soc. \textbf{52}(1975), 293-296.

\bibitem[O82]{Ohsawa82}
T. Ohsawa, \emph{Isomorphism theorems for cohomology groups of weakly $1$-complete manifolds},
Publ. Res. Inst. Math. Sci.  \textbf{18}  (1982), no. 1, 191--232.

\bibitem[O88]{Ohsawa88}
Takeo Ohsawa, \emph{Complete K\"{a}hler manifolds and function
theory of several complex variables}, Sugaku Expositions
\textbf{1} (1988), 75-93.

\bibitem[Se53]{Serre53}
J.-P. Serre, \emph{Quelques probl\`{e}mes globaux relatifs aux
vari\'{e}t\'{e}s de Stein}, Colloque sur les Fonctions de
Plusieurs Variables, 57-68, Brussels, 1953.

\bibitem[Sh92]{Shaw92}
Mei-Chi Shaw, \emph{Local existence theorems with estimates for
  {$\overline\partial\sb b$} on weakly pseudo-convex {CR} manifolds},
  Math. Ann. \textbf{294} (1992), no.~4, 677--700.

\bibitem[Sh09]{Shaw09}
\bysame, \emph{The closed range property for $\bar\partial$ on domains
with pseudoconcave boundary}, preprint, 2009, in this volume.

\bibitem[Siu67]{Siu67}
Yum-Tong Siu, \emph{Non-countable dimensions of cohomology groups of analytic sheaves and domains of holomorphy}, Math. Z. \textbf{102} (1967), 17-29.

\bibitem[Sk72]{Skoda72}
Henri Skoda, \emph{Application des techniques $L\sp{2}$ \`{a} la th\'{e}orie des id\'{e}aux d'une alg\`{e}bre de fonctions holomorphes avec poids}, Ann. Sci. \'{E}cole Norm. Sup. (4)  \textbf{5} (1972), 545-579.

\bibitem[St09]{Straube09}
Emil Straube, \emph{Lectures on the $L^2$-Sobolev theory of the $\dbar$-Neumann problem}, to appear in the series ESI Lectures in Mathematics and Physics, European Math. Soc.

\bibitem[W84]{Wiegerinck84}
Jan Wiegerinck, \emph{Domains with finite dimensional Bergman space}, Math. Z. \textbf{187} (1984), 559-562.

\end{thebibliography}
\providecommand{\bysame}{\leavevmode\hbox
to3em{\hrulefill}\thinspace}

\end{document}